\documentclass[12pt,a4paper,oneside,titlepage]{article}

\usepackage{amsmath}
\usepackage{amsthm}
\usepackage{amsfonts}
\usepackage[utf8]{inputenc}
\usepackage{graphicx}
\usepackage{geometry}
\usepackage{color}
\newgeometry{tmargin=2cm, bmargin=3cm, lmargin=2cm, rmargin=2cm}
\usepackage{ifthen}
\usepackage{hyperref}
\usepackage{nameref}

\makeatletter
\let\orgdescriptionlabel\descriptionlabel
\renewcommand*{\descriptionlabel}[1]{%
  \let\orglabel\label
  \let\label\@gobble
  \phantomsection
  \edef\@currentlabel{#1}%
  \let\label\orglabel
  \orgdescriptionlabel{#1}%
}
\makeatother

\newtheorem{theorem}{Theorem}[section]
\newtheorem{lemma}[theorem]{Lemma}
\newtheorem{definition}[theorem]{Definition}
\newtheorem{problem}{Problem}
\newtheorem*{problem2}{Problem}
\newtheorem{remark}[theorem]{Remark}

\newcommand{\setR}{\mathbb{R}}
\newcommand{\Nset}{\mathbb{N}}

\renewcommand{\epsilon}{\varepsilon}

\newcommand{\bracket}[1]{\left( {#1} \right)}
\newcommand{\set}[1]{\left\lbrace {#1} \right\rbrace} 
\newcommand{\abs}[1]{\left| {#1} \right|} 				
\newcommand{\dt}[1][t]{\,\mathrm{d}{#1}}				
\newcommand{\essinf}{\operatorname{essinf}\limits}

\newcommand{\norm}[2][]{ \left\| {#2} \right\|_{#1}}
\newcommand{\dual}[3][]{ \left\langle {#2} ; {#3} \right\rangle_{ {#1}}}
\newcommand{\distance}[3][d]{ \operatorname{{#1}}\left( {#2} ; {#3} \right)}
\newcommand{\sequence}[3]{\left( {#1} \right)_{#2}^{#3} }
\newcommand{\dualSpace}[1]{{#1}^\ast}
\newcommand{\Ck}[2][]{\operatorname{C}^{#1} \ifthenelse{\equal{#2}{}}{}{\left( {#2} \right)}}
\newcommand{\Ckz}[2][]{\operatorname{C}^{#1}_0 \ifthenelse{\equal{#2}{}}{}{\left( {#2} \right)}}
\newcommand{\Lp}[2][p]{\operatorname{L}^{#1}\ifthenelse{ \equal{#2}{} }{}{\left( {#2} \right)}}
\newcommand{\Wppz}[3]{\operatorname{W}^{ {#1},{#2}}_{0} \ifthenelse{ \equal{#3}{} }{}{\left( {#3} \right)}}
\newcommand{\Wpp}[3]{\operatorname{W}^{ {#1},{#2}} \ifthenelse{ \equal{#3}{} }{}{\left( {#3} \right)} }
\newcommand{\Hpz}[2]{\operatorname{H}^{{#1}}_{0} \ifthenelse{ \equal{#2}{} }{}{\left( {#2} \right)}}

\newcommand{\pLaplace}{\Delta_p}
\newcommand{\boundary}{\partial}
\newcommand{\cut}[2]{ \left. {#1} \right|_{#2} }
\newcommand{\weakto}{\rightharpoonup}
\newcommand{\compactembedding}{\subset \subset}
\newcommand{\continuousembedding}{\hookrightarrow}

\author{Kowalski Piotr, Piwnik Joanna} 
\title{Three solutions for elliptic Dirichlet boundary value problem with singular weight.\\ }
\date{\today}

\begin{document}
\maketitle

\newpage
\begin{abstract}
In this paper we prove the~existence of two non-trivial weak solutions of Dirichlet boundary value problem for p-Laplacian problem with a~singular part and two disturbances satisfying the~proper assumptions. The~abstract existence result we use is the~famous Ricceri theorem.
\end{abstract}

\section{Introduction}

Singular elliptic problems have been intensively and widely studied in recent years. Among others, p-Laplacian operator appears to be mostly investigated elliptic operator. An important subclass of such problems are problems involving singular nonlinearities \cite{Kristaly.Varga}. We recall that $\Delta_p$ denotes the~p-Laplacian operator, namely $\Delta_p u=\text{div}\bracket{|\nabla  u|^{p-2} \nabla u}$. In this paper we are especially interested in eigenvalue problem, derived from the~one well explained by Lindquist in \cite{Lindqvist}. Some problems require to find the~smallest positive scalar $\lambda_1 \in \setR$ for which the~equation  
\begin{equation}\label{intro.eq.1}
- \pLaplace u(x) = \lambda_1 a(x) \abs{u(x)}^{p-2}u(x) \text{  for  a.e.  }x \in \Omega,
\end{equation}
with $\Omega \subseteq \setR^n,\ n\geq 2,$ being bounded with sufficiently smooth boundary, has a~nontrivial solution in $\Wppz{1}{p}{\Omega}$, see \cite{Cuesta.Quoirin, Lucia.Prashanth}. Such $\lambda_1$ shall be referred as the~principle eigenvalue. The~starting point is usually the~simplest case $a(x)\equiv 1$, but also the~weighted version of this equation finds a~lot of applications. By the~solution we usually understand the~weak one. Some researchers are interested in proving the~existence of three solutions for small $\mu, \lambda_1>\mu >0$ with a~semilinear term. For example
\begin{equation}\label{intro.eq.2.1}
- \pLaplace u(x) = \mu a(x) \abs{u(x)}^{p-2}u(x) + f(x,u(x))\text{  for a.e.  }x \in \Omega,
\end{equation}
or
\begin{equation}\label{intro.eq.2}
- \pLaplace u(x) + \mu a(x) \abs{u(x)}^{p-2}u(x) = f(x,u(x))\text{  for a.e.  }x \in \Omega,
\end{equation}
where $a \in \Lp[\infty]{\Omega}$ with $\essinf\limits_{x \in \Omega} a(x) >0$, e.q, \cite{Bonanno.Candito} or \cite{Dagui.Molica-Bisci}. We should note that the~problem \eqref{intro.eq.2} is essentially different from the~problem in \eqref{intro.eq.1}, since both $\lambda_1>0$ and $\mu>0$.

\noindent Yet the~most interesting cases contain the~weight function $a\colon \setR^n \to \setR$ unbounded and having a~singularity. Stationary problem involving such nonlinearities describes some applied economical models and several physical phenomena, for instance conduction in electrically conducting materials. The~problem \eqref{intro.eq.2.1} can be found for example in \cite{Cuesta.Quoirin,Ferrara.Molica-Bisci,Filippucci.Pucci.Robert}. In particular, many results on the~existence and multiplicity of solutions for nonlinear problems involving the~$p$-Laplacian operator have been obtained by adopting variational methods. We also recall that, several authors have treated the~case $p\leq n$, by using completely different techniques \cite{Binding.Drabek.Huang, Gasinski.Papageorgiou}.

\noindent In this paper we consider the~following problem defined on $\Omega$, where $\Omega\subseteq \setR^n$ is bounded, has Lipschitz boundary, $0 \in \Omega$ and $2 \leq p < n$. 
By $\Wppz{1}{p}{\Omega}$ we understand the~closure of $\Ckz[\infty]{\Omega}$ in norm of Sobolev space $\Wpp{1}{p}{\Omega}$. 
\begin{problem} \label{problem.classical.solution}
Find $u \in \Wppz{1}{p}{\Omega}$ such that:
\begin{align*}
- \pLaplace u(x) + \mu \frac{\abs{u(x)}^{p-2}u(x)}{\abs{x}^p} &= \lambda f(u(x)) + \gamma g(u(x)),\: \mbox{  for a.e.  } x \in \Omega, \\
\cut{u}{\boundary \Omega} &\equiv 0.
\end{align*}
\end{problem}
\noindent The~above presented problem is understand as equivalent to Problem \ref{intro.problem.weak.solution}.
\begin{problem} \label{intro.problem.weak.solution}
Find $u \in \Wppz{1}{p}{\Omega}$ such that for all $v \in \Wppz{1}{p}{\Omega}$
\begin{align*}
\int\limits_{\Omega} \abs{\nabla u(x)}^{p-2} \nabla u(x) \nabla v(x) + \mu \frac{\abs{u(x)}^{p-2} u(x) v(x) }{ \abs{x}^p} \dt[x] \\
=\lambda \int\limits_{\Omega} f(u(x))v(x) \dt[x] + \gamma \int\limits_{\Omega} g(u(x))v(x) \dt[x].
\end{align*}
\end{problem}
\setcounter{problem}{1}

\noindent We assume that $f\colon \setR \to \setR$ satisfies the~following conditions 
\begin{description}
\item[{(f1)}] $\lim\limits_{\abs{t}\to 0}\frac{f(t)}{\abs{t}^{p-1}}=0$. 
\item[{(f2)}] $\lim\limits_{\abs{t}\to +\infty}\frac{f(t)}{\abs{t}^{p-1}}=0$. 
\item[{(f3)}] $\sup\limits_{t\in \setR} F(t)>0$, where $F(t)=\int\limits^{t}_{0} f(s) \dt[s]$.
\end{description}
For function $g\colon \setR \to \setR$ we assume only that there exists $c_g\geq 0$ and $1<q<\frac{pn}{n-p}$ such that 
$$
\abs{g(t)}\leq c_g (1+\abs{t}^{q-1}) \text{ for all } t\in\setR\ .
$$
The main tool shall be the~Ricceri three critical point theorem \cite{Ricceri}, and from it shall follow that there are intervals for $\lambda>0$ and $\gamma >0$ in which the~Problem \ref{intro.problem.weak.solution} has at least three weak solutions.

This paper is organized as follows. In section 2 we recall some results from the~theory of uniformly monotone operators, the~properties of uniformly convex Banach spaces and some inequalities related to embedding results. In section 3 we introduce the~corresponding energy functional for Problem \ref{intro.problem.weak.solution} and we define some auxiliary functionals. We also study properties of those functionals, such as sequential weak lower semicontinuity, the~compactness of derivative and existence of continuous inverse. In section 4 we prove the~necessity conditions for a~certain three critical points theorem by Ricceri and we also formulate the~main result of this paper. The last section contains an example.
This work is mainly motivated by the~study of elliptic problems with singular and sublinear potentials done in \cite{Kristaly.Varga}.

\section{Preliminaries}
We recall some of the~basic definitions and theorems we require from functional analysis. We start by recalling the~definition of compact operator.

\begin{definition}[{\cite[def 4.16]{Rudin} Compact operator}]
Assume $X,Y$ are Banach spaces, $U$ is open unit ball contained in $X$. Operator $A\colon X\to Y$ is called compact if closure of $A(U)$ is compact in $Y$. 
\end{definition}
\noindent We shall prove that a~functional satisfies a~strong monotonic condition, namely it is an uniformly monotone operator.

\begin{definition}[{\cite[def 25.2]{Zeidler} Uniformly monotone operator}]
Let $X$ be a~Banach space and $A\colon X\to \dualSpace{X}$. We say that A is uniformly monotone operator, when
$$
a\bracket{\norm{u-v}{}}\norm{u-v}\leq \dual{Au-Av}{u-v} \text{ for all } u,v \in X,
$$
where the~continuous function $a\colon\setR_+\to \setR_+$, is strictly monotone increasing with $a(0)=0$ and $\lim\limits_{t\to +\infty}a(t)=+\infty.$
\end{definition}

\noindent The~Ricceri abstract existence result require that derivative of a~given operator admits a~continuous inverse. In order to obtain such property we shall use the~following theorem. 

\begin{theorem}[{\cite[th 26.A]{Zeidler} Browder-Minty theorem}]\label{Browder-Minty}
Let $A\colon X\to X^*$ be a~monotone, coercive, and hemicontinuous operator on the~real, separable, reflexive Banach space $X$. If $A$ is strictly monotone, then the~inverse operator $A^{-1}\colon X^*\to X$ exists and is strictly monotone, demicontinuous, and bounded. If $A$ is uniformly monotone, then $A^{-1}$ is continuous. 
\end{theorem}

\noindent We denote that if $A$ is uniformly monotone, then $A$ is strictly monotone, coercive and hemicontinuous.

\noindent The~following well known analytic inequality shall be helpful.
\begin{lemma}[{ \cite[page 3]{Chabrowski}}]\label{Chabrowski}
Let $p\geq2.$ For each $x,y\in \setR^n$ there occurs 
$$
(\abs{x}^{p-2} x-\abs{y}^{p-2}y)\cdot (x-y)\geq \frac{2}{p(2^{p-1}-1)}\abs{x-y}^p.
$$
\end{lemma}

\begin{theorem}[{\cite[th 6.30]{Adams} Poincar\`e inequality}]\label{Poincare inequality}
If domain $\Omega\subset \setR^n$ is bounded, then there exists a~constant $C_p$ dependent only on $\Omega$ and $p$, such that for all $u \in \Ckz[\infty]{\Omega}$
$$
\norm[{\Lp[p]{\Omega}}]{u} \leq C_p \norm[{\Lp[p]{\Omega;\setR^n}}]{\nabla u}. 
$$
\end{theorem}

\noindent Since Poincar\`e inequality holds, in this paper, we shall use $\Wppz{1}{p}{\Omega}$ given with a~norm equivalent to Sobolev's one, i.e. 
$$
\norm[{\Wppz{1}{p}{\Omega}}]{x}=\norm[{\Lp[p]{\Omega}}]{\nabla x}.
$$

\begin{theorem}[{\cite[th 6.3]{Adams} Rellich–Kondrachov theorem}]\label{preliminaries.theorem.rellich-kondrachov}
Let $\Omega \subset \setR^n$ be an open, bounded Lipschitz domain, and let $1 \leq p <n$. Set $p^{\ast} = \frac{np}{n-p}$. Then the~Sobolev space $\Wppz{1}{p}{\Omega}$ is continuously embedded in $\Lp[p{^\ast}]{\Omega}$ and compactly embedded in $\Lp[q]{\Omega}$, where $1 \leq q < p^\ast$.
\begin{equation*}
\Wppz{1}{p}{\Omega} \continuousembedding \Lp[p^{\ast}]{\Omega},
\end{equation*}
and
\begin{equation*}
\Wppz{1}{p}{\Omega} \compactembedding \Lp[q]{\Omega} \mbox{ for } 1\leq q < p^\ast.
\end{equation*}
\end{theorem}

\noindent By above theorem we immediately obtain the~following:
$$
\Wppz{1}{p}{\Omega}\compactembedding \Lp[p]{\Omega}.
$$ 

\noindent The~theory of Ricceri is easily applicable to problems connected to a~norm that is an uniformly convex one. Thus we recall some classical results of Clarkson \cite{Clarkson}.
\begin{definition}[{\cite[def 1]{Clarkson}} Uniformly convex space]
A Banach space $B$ will be said to be uniformly convex if to each $\epsilon$, $0<\epsilon\leq 2$, there corresponds a~$\delta(\epsilon)>0$ such that the~conditions
$$
\norm{x}=\norm{y}=1,\quad \norm{x-y}\geq \epsilon
$$
imply
$$
\norm{\frac{x+y}{2}}\leq 1-\delta(\epsilon).
$$
\end{definition}
In this paper we shall use two uniformly convex spaces, $\Wppz{1}{p}{\Omega}$ and $\Lp[p]{\abs{x}^{-p}; \Omega}$, simultaneously. Thus the~following result concerning the~product of uniformly convex spaces is required.

\begin{definition}[{\cite{Clarkson}} Uniformly convex product]
Let $N(a_1, a_2, \cdots, a_k)$ be a~non-negative continuous function of the~non-negative variables $a_i$. We say that $N$ is
\begin{enumerate}
\item homogeneous, if for $c \geq 0$
$$
N(ca_1, ca_2, \cdots , ca_k)=cN(a_1,a_2,\cdots, a_k);
$$
\item strictly convex, if
$$
N(a_1+b_1, a_2+b_2, \cdots, a_k+b_k)<N(a_1,a_2,\cdots, a_k)+N(b_1, b_2, \cdots, b_k),
$$
unless $a_i=cb_i\: (i=1,2,\cdots, k)$. In the~latter case we have equality by condition 1.
\item strictly increasing, if it is strictly increasing in each variable separately. 
\end{enumerate}
A familiar example of a~function N satisfying these conditions is 
$$
N(a_1,\ldots,a_k)=\bracket{\sum\limits^k_{i=1} a_i^p}^{\frac{1}{p}}\: (p>1);
$$ 
here condition 2. becomes the~inequality of Minkowski.

\noindent Suppose now that a~finite number of Banach spaces $B_1, B_2, \cdots, B_k$ are given, and that $B$ is their product. We shall call B a~\textit{uniformly convex product} of $B_i$ if the~norm of an element $x=(x^1, x^2, \cdots, x^k)$ of $B$ is defined by
$$
\norm[]{x}=N(\norm[]{x^1},\norm[]{x^2}, \cdots, \norm[]{x^k}),
$$ 
where $N$ is a~continuous non-negative function satisfying the~conditions 1-3.
\end{definition}

\begin{theorem}[{\cite{Clarkson} Clarkson theorem}]\label{Clarkson}
The uniformly convex product of a~finite number of uniformly convex Banach spaces is a~uniformly convex Banach space.
\end{theorem}

\noindent We shall also need a~relation in between Sobolev space $\Wppz{1}{p}{\Omega}$ and weighted $\Lp[p]{\abs{x}^{-p},\Omega}$ space.

\begin{theorem}[{\cite{Azorero} Hardy inequality}]
Assume $1<p<n$ and $u\in \Wppz{1}{p}{\Omega}$, then \\\mbox{$u \in \Lp[p]{\abs{x}^{-p};\Omega}$} and
$$
\int\limits_{\Omega}\frac{\abs{u(x)}^p}{\abs{x}^p} \dt[x] \leq C_{n,p} \int\limits_{\Omega}\abs{\nabla u(x)}^p \dt[x],
$$
with $C_{n,p}=\bracket{\frac{p}{n-p}}^p$.
\end{theorem}

\noindent Finally we shall also require the~following topological lemma:
\begin{lemma}\label{metric lemma}
Let $d_1, d_2$ be metrics on metric space $X$. Then $d(x,y)=\sqrt[p]{d_1(x,y)^p+d_2(x,y)^p}$ is also a~metric on $X$.
\end{lemma}
\begin{proof}
We shall prove metric axioms are satisfied.
\begin{itemize}
\item[M1)] Let $x=y$. Then $d(x,y)=d(x,x)=\sqrt[p]{d_1(x,x)^p+d_2(x,x)^p}=\sqrt[p]{0^p+0^p}=\sqrt[p]{0}=0$. Conversely $\distance{x}{y}=0$ then $\distance[d_1]{x}{y}=0$. Hence $x=y$.
\item[M2)] $d(x,y)=\sqrt[p]{d_1(x,y)^p+d_2(x,y)^p}=\sqrt[p]{d_1(y,x)^p+d_2(y,x)^p}=d(y,x)$. Hence $d$ is symmetric.
\item[M3)] Let $x,y,z\in X$. Then 
\begin{align*}
d(x,z)& =\sqrt[p]{d_1(x,z)^p+d_2(x,z)^p}\\
	  &	\leq \sqrt[p]{(d_1(x,y)+d_1(y,z))^p+(d_2(x,y)+d_2(y,z))^p}\\
      & =\norm[{\Lp[p]{R^2}}]{(d_1(x,y),d_2(x,y))+(d_1(y,z),d_2(y,z))}\\
      & \leq  \norm[{\Lp[p]{R^2}}]{d_1(x,y),d_2(x,y)}+ \norm[{\Lp[p]{R^2}}]{d_1(y,z),d_2(y,z)}\\
      & =\sqrt[p]{d_1(x,y)^p+d_2(x,y)^p}+\sqrt[p]{d_1(y,z)^p+d_2(y,z)^p}=d(x,y)+d(y,z).      
\end{align*}
\end{itemize} 
Thus $d$ is a~metric.
\end{proof}

\section{Variational Framework}
In this paper we prove the~existence of weak solutions of problem corresponding to Problem \ref{problem.classical.solution}, namely:
\begin{problem}\label{problem.weak.solution}
Find $u \in \Wppz{1}{p}{\Omega}$ such that for all $v \in \Wppz{1}{p}{\Omega}$
\begin{align*}
\int\limits_{\Omega} \abs{\nabla u(x)}^{p-2} \nabla u(x) \nabla v(x) + \mu \frac{\abs{u(x)}^{p-2} u(x) v(x)}{\abs{x}^p}\dt[x]\\
=\lambda \int\limits_{\Omega} f(u(x)) v(x)\dt[x] + \gamma \int\limits_{\Omega} g(u(x)) v(x)\dt[x],
\end{align*}
where $\Omega\subseteq \setR^n$ is bounded, has Lipschitz boundary, $0 \in \Omega$ and $2 \leq p < n$.
\end{problem}

\noindent We assume that $\mu \in(0,+\infty)$. We also require certain conditions on continuous functions \\$f,g \colon \setR \to \setR$, namely
\begin{description}
\item[{(f1)}\label{condition.f1}] $\lim\limits_{\abs{t}\to 0}\frac{f(t)}{\abs{t}^{p-1}}=0$. 
\item[{(f2)}\label{condition.f2}] $\lim\limits_{\abs{t}\to +\infty}\frac{f(t)}{\abs{t}^{p-1}}=0$. 
\item[{(f3)}\label{condition.f3}] $\sup\limits_{t\in \setR} F(t)>0$, where $F(t)=\int\limits^{t}_{0} f(s) ds$.
\item[{(g1)}\label{condition.g1}] There exists $c_g$ and $1<q<\frac{pn}{n-p}$ such that 
$$
\abs{g(t)}\leq c_g (1+\abs{t}^{q-1}) \text{ for all } t\in\setR\ .
$$
\end{description}

\noindent We define the~following functionals, $\Phi, J_1, J_2, E\colon  \Wppz{1}{p}{\Omega} \to \setR$ given by the~formulas
\begin{align*}
& \Phi(u)=\frac{1}{p} \int\limits_{\Omega} \abs{\nabla u(x)}^p + \mu \frac{\abs{u(x)}^p}{\abs{x}^p} \dt[x], \\
& J_1(u)=\int\limits_{\Omega} F(u(x)) \dt[x],\\
& J_2(u)=\int\limits_{\Omega} G(u(x)) \dt[x],\\
& E(u)=\Phi (u)-\lambda J_1(u)-\gamma J_2(u),
\end{align*}
where $G(t)=\int\limits^{t}_{0} g(s) ds$. 

We start by proving that any critical point of $E$ is a~weak solution of Problem \ref{problem.weak.solution}.
\begin{lemma}\label{well_defined_functionals}
Assume, that conditions \ref{condition.f1}, \ref{condition.f2} and \ref{condition.g1} hold. Then functionals $\Phi$, $J_1$ and $J_2$ are well defined and they have the~following G\^ateaux derivatives:
\begin{align*}
& \dual{\Phi'(u)}{v}=\int\limits_{\Omega} \abs{\nabla u(x)}^{p-2} \nabla u(x) \nabla v(x) + \mu \frac{\abs{u(x)}^{p-2} u(x) v(x)}{\abs{x}^p} \dt[x], \\
& \dual{J_1'(u)}{v}= \int\limits_{\Omega} f(u(x)) v(x) \dt[x], \\
& \dual{J_2'(u)}{v}= \int\limits_{\Omega} g(u(x)) v(x) \dt[x].
\end{align*}
\end{lemma}
\begin{proof}
All of the~functional are well-defined.\\
\noindent We calculate the~G\^ateaux derivative:
\begin{align*}
\dual{\Phi'(u)}{v}=& \lim\limits_{\lambda \downarrow 0} \frac{\frac{1}{p} \int\limits_{\Omega} \abs{\nabla \bracket{u(x)+ \lambda v(x)}}^p + \mu \frac{\abs{u(x)+\lambda v(x)}^p}{\abs{x}^p}-\abs{\nabla u(x)}^p-\mu \frac{\abs{u(x)}^p}{\abs{x}^p}\dt[x]}{\lambda}\\
=&\frac{1}{p} \int\limits_{\Omega} \lim\limits_{\lambda \downarrow 0} \frac{\abs{ \nabla u(x)+\lambda\nabla v(x))}^p -\abs{\nabla u(x)}^p}{\lambda} +\frac{\mu}{\abs{x}^p}\lim\limits_{\lambda\downarrow 0}\frac{\abs{u(x) + \lambda v(x)}^p - \abs{u(x)}^p}{\lambda}\dt[x]\\
=& \frac{1}{p}\int\limits_{\Omega} p \abs{\nabla u(x)}^{p-2} \nabla u(x) \nabla v(x) + \frac{\mu}{\abs{x}^p}  p  \abs{u(x)}^{p-2} u(x) v(x) \dt[x].
\end{align*}
The following equality holds due to Lebesgue's dominated convergence theorem.
\begin{align*}
\dual{J_1'(u)}{v}=&\lim\limits_{\lambda \downarrow 0} \frac{\int\limits_{\Omega} F(u(x)+\lambda v(x))-F(u(x))\dt[x]}{\lambda}= \int\limits_{\Omega}\lim\limits_{\lambda \downarrow 0} \frac{ F(u(x)+\lambda v(x))-F(u(x))\dt[x]}{\lambda}\\
=&\int\limits_{\Omega} \lim\limits_{v(x)\lambda\to 0} \frac{F(u(x)+\lambda v(x))-F(u(x))}{\lambda v(x)}\dt[x]=\int\limits_{\Omega} f(u(x)) v(x) \dt[x].
\end{align*}
The dominated convergence theorem is applicable since,
we can estimate the function $\frac{F(u(x)+\lambda v(x))-F(u(x))}{\lambda}$ from above in a following way
\begin{align*}
\abs{f_\lambda (x)}
=&
\abs{\frac{F(u(x)+\lambda v(x))-F(u(x))}{\lambda}}
=
\abs{\frac{\int\limits_{0}^{u(x)+\lambda v(x)}f(s)\dt[s]-\int\limits_{0}^{u(x)} f(s) \dt[s]}{\lambda}}\\
\leq &
\frac{\abs{\int\limits_{u(x)}^{u(x)+\lambda v(x)}\abs{f(s)}\dt[s]}}{\lambda} 
\leq 
\frac{\abs{\int\limits_{u(x)}^{u(x)+\lambda v(x)}\abs{c_f\abs{s}^{p-1}}\dt[s]}}{\lambda}
=
\frac{c_f}{\lambda}\abs{\int\limits_{u(x)}^{u(x)+\lambda v(x)}\abs{\abs{s}^{p-1}}\dt[s]} \\
\leq &
\frac{c_f}{\lambda}\abs{\int\limits_{u(x)}^{u(x)+\lambda v(x)}\sup\limits_{t \in [u(x),u(x)+\lambda v(x)]}\abs{\abs{t}^{p-1}}\dt[s]}
\leq
\frac{c_f}{\lambda}\abs{\int\limits_{u(x)}^{u(x)+\lambda v(x)}\sup\limits_{t \in [u(x),u(x) + v(x)]} \abs{ \abs{t}^{p-1} }\dt[s]}\\
\leq&
\frac{c_f}{\lambda} \max\{\abs{u(x)}^{p-1}, \abs{u(x)+v(x)}^{p-1}\} \int\limits_{u(x)}^{u(x)+\lambda v(x)} 1 \dt[s]\\
=&
\frac{c_f}{\lambda} \max\{\abs{u(x)}^{p-1}, \abs{u(x)+v(x)}^{p-1}\} \lambda v(x) 
=
 c_f \max\{\abs{u(x)}^{p-1}, \abs{u(x)+v(x)}^{p-1}\} v(x).
\end{align*}
It follows from H\"older inequality that this function is integrable.
Calculation for the functional $J_2$ are very similar:
\begin{align*}
\dual{J_2'(u)}{v}=&\lim\limits_{\lambda \downarrow 0} \frac{\int\limits_{\Omega} G(u(x)+\lambda v(x))-G(u(x))\dt[x]}{\lambda}= \int\limits_{\Omega}\lim\limits_{\lambda \downarrow 0} \frac{ G(u(x)+\lambda v(x))-G(u(x))\dt[x]}{\lambda}\\
=&\int\limits_{\Omega} \lim\limits_{v(x)\lambda\to 0} \frac{G(u(x)+\lambda v(x))-G(u(x))}{\lambda v(x)}\dt[x]=\int\limits_{\Omega} g(u(x)) v(x) \dt[x],
\end{align*}
since the following estimate holds: $\frac{G(u(x)+\lambda v(x))-G(u(x))}{\lambda}$ from above:
\begin{align*}
\abs{g_\lambda (x)}
=&
\abs{\frac{G(u(x)+\lambda v(x))-G(u(x))}{\lambda}}
=
\abs{\frac{\int\limits_{0}^{u(x)+\lambda v(x)}g(s)\dt[s]-\int\limits_{0}^{u(x)} g(s) \dt[s]}{\lambda}}\\
\leq &
\frac{\abs{\int\limits_{u(x)}^{u(x)+\lambda v(x)}\abs{g(s)}\dt[s]}}{\lambda} 
\leq 
\frac{\abs{\int\limits_{u(x)}^{u(x)+\lambda v(x)}\abs{c_g\bracket{ 1 +\abs{s}^{q-1}}} \dt[s]}}{\lambda}
=
\frac{c_g}{\lambda}\abs{\int\limits_{u(x)}^{u(x)+\lambda v(x)}1+\abs{s}^{q-1}\dt[s]} \\
\leq &
\frac{c_g}{\lambda}\abs{\int\limits_{u(x)}^{u(x)+\lambda v(x)} 1+ \sup\limits_{t \in [u(x),u(x)+\lambda v(x)]}\abs{t}^{q-1}\dt[s]}\\
\leq&
\frac{c_g}{\lambda}\abs{\lambda v(x) + \int\limits_{u(x)}^{u(x)+\lambda v(x)}\sup\limits_{t \in [u(x),u(x) + v(x)]} \abs{t}^{q-1} \dt[s]}\\
\leq&
c_g \bracket{\abs{v(x)}+\max\{\abs{u(x)}^{q-1}, \abs{u(x)+v(x)}^{q-1}\}v(x)}.
\end{align*}
We can show, using Theorem \ref{preliminaries.theorem.rellich-kondrachov} and H\"older inequality, that the obtained dominant is integrable. 
Which concludes the proof.
\end{proof}
\begin{remark}
From the~proof of Lemma \ref{well_defined_functionals}, we can conclude that $\Phi$ is bounded on bounded subsets of $\Wppz{1}{p}{\Omega}$.
\end{remark}

\noindent We can note that using the~functional given above we can rewrite Problem \ref{problem.weak.solution} also as:\\
Find $u\in \Wppz{1}{p}{\Omega}$ such that
$$
\dual{E'(u)}{v}= \dual{\Phi'(u)}{v} - \lambda\dual{ J_1'(u)}{v}-\gamma \dual{ J_2'(u)}{v}=0,
$$
for all $v \in \Wppz{1}{p}{\Omega}$.

\noindent Thus, any critical point to $E$ is a~solution of Problem \ref{problem.weak.solution}.
At first we shall concentrate on properties of functional $\Phi$. We shall prove that it has many similarities to a~norm.\\
We define the following functional for every $u \in \Wppz{1}{p}{\Omega}$
\begin{align*}
u\mapsto \norm[{*}]{u}=\bracket{p\cdot \Phi(u)}^{\frac{1}{p}}=\sqrt[p]{\int\limits_{\Omega}\abs{\nabla u(x)}^p+\mu \frac{\abs{u(x)}^p}{\abs{x}^p}\dt[x]}.
\end{align*}

\begin{lemma}
Function $\Wppz{1}{p}{\Omega} \ni u \mapsto \norm[{*}]{u}$ is a~norm on the~space $\Wppz{1}{p}{\Omega}$.
\end{lemma}
\begin{proof}
We observe that $\norm[*]{u}=(p\cdot \Phi (u))^{\frac{1}{p}}$. We will show $\norm[{*}]{u}$ satisfies norm axioms.
\begin{itemize}
\item[N1)] Let $\norm[{*}]{u}=0$. Then $(p\cdot \Phi (u))^{\frac{1}{p}}=0$, thus $p\cdot\Phi(u)$=0. We note explicitly
$$
\int\limits_{\Omega} \abs{\nabla u(x)}^p + \mu \frac{\abs{u(x)}^p}{\abs{x}^p}\dt[x]=0.
$$
Since $p$-power of modulus is non-negative, then $\nabla u(x)\equiv 0$ and $u(x)\equiv 0$ almost everywhere in $\Omega$, hence $u=0$ in $\Wppz{1}{p}{\Omega}$. The~opposite implication holds instantly.
\item[N2)]Let $\alpha \in \setR.$ Then: 
\begin{align*}
\norm[{*}]{\alpha u}&=(p\cdot \Phi (\alpha u))^{\frac{1}{p}}\\
&=\bracket{\int\limits_{\Omega} \abs{\nabla(\alpha u(x))}^p + \mu \frac{\abs{\alpha u(x)}^p}{\abs{x}^p }\dt[x]}^{\frac{1}{p}}\\
&=\abs{\alpha}\bracket{\int\limits_{\Omega}\abs{\nabla(u(x))}^p + \mu \frac{\abs{u(x)}^p}{\abs{x}^p}\dt[x]}^{\frac{1}{p}}\\
&=\abs{\alpha} (p\cdot \Phi (u))^{\frac{1}{p}}= \abs{\alpha}\norm[{*}]{u}.
\end{align*}
\item[N3)] Define two metrics $d_1, d_2\colon  \Wppz{1}{p}{\Omega}\times \Wppz{1}{p}{\Omega} \to \setR$ as 
\begin{align*}
d_1(x,y)&=\norm[{\Wppz{1}{p}{\Omega}}]{x-y},\\
d_2(x,y)&=\sqrt[p]{\mu} \norm[{\Lp[p]{\abs{x}^{-p};\Omega}}]{x-y}.
\end{align*}
Then it follows from by Lemma \ref{metric lemma} that $d(x,y)=\sqrt[p]{d_1(x,y)^p+d_2(x,y)^p}$ is also a~metric. Thus, for $x=u+w, z=0, y=w$ we have: 
\begin{align*}
\norm[{*}]{u+w}=&\bracket{p\cdot \Phi (u+w)}^{\frac{1}{p}} \\
=& \sqrt[p]{d_1(u+w,0)^p+d_2(u+w,0)^p}\\
=&d(u+w,0)\leq d(u+w,w)+d(w,0)\\
=&\sqrt[p]{\norm[{\Wppz{1}{p}{\Omega}}]{u+w-w}^p+\mu \norm[{\Lp[p]{\abs{x}^{-p};\Omega}}]{u+w-w}^p}\\
&+\sqrt[p]{\norm[{\Wppz{1}{p}{\Omega}}]{w}^p+\mu \norm[{\Lp[p]{|x|^{-p};\Omega}}]{w}^p}\\
=&(p\cdot \Phi (u))^{\frac{1}{p}}+(p\cdot \Phi (w))^{\frac{1}{p}}=\norm[{*}]{u} + \norm[{*}]{w}.
\end{align*}
\end{itemize}
Hence $u \mapsto \norm[{*}]{u}$ is a~norm.
\end{proof}

\begin{lemma}\label{swlsc}
Functional $\Phi$ is sequentially weakly lower semicontinuous.
\end{lemma}

\begin{proof}
From its definition it follows that
$$
\Phi(u)= \frac{1}{p} \bracket{\norm[{\Wppz{1}{p}{\Omega}}]{u}^p+\mu\cdot \norm[{\Lp[p]{|x|^{-p},\Omega}}]{u}^p}.
$$
Since $\norm[{\Wppz{1}{p}{\Omega}}]{\cdot}$ and $\norm[{\Lp[p]{|x|^{-p},\Omega}}]{\cdot}$ are sequentially weakly lower semicontinuous, then $\Phi$ is sequentially weakly lower semicontinuous.
\end{proof}

\begin{lemma}\label{coercive}
Operator $\Phi$  is coercive.
\end{lemma}

\begin{proof}
$$
\Phi(u)=\frac{1}{p}\int\limits_{\Omega} \abs{\nabla u(x)}^p +\mu \frac{\abs{u(x)}^p}{\abs{x}^p}\dt[x]\geq \frac{1}{p}\int\limits_{\Omega} \abs{\nabla u(x)}^p =\frac{1}{p}\norm[\Wppz{1}{p}{\Omega}]{u}^p \xrightarrow{\norm{u}\to\infty}\infty.
$$
Thus, $\Phi$ is coercive operator.
\end{proof}

\noindent Now we focus on properties of $J_1$ and $J_2$.

\begin{lemma}\label{compact_dericative_J1}
Assume \ref{condition.f1}-\ref{condition.f2}. Then functional $J_1$ has a~compact derivative.
\end{lemma}

\begin{proof}
We split the~proof into two parts. At first we shall prove that the~image of a~bounded set through 
$J_1'$ is bounded, afterwards we shall prove the~existence of a~convergent subsequence within this image. We take a~bounded sequence $(u_n)_{n \in \Nset}\subset \Wppz{1}{p}{\Omega}$, which means there exists $M>0$ such that for all $n \in \Nset$
$$
\norm[{\Wppz{1}{p}{\Omega}}]{u_n} \leq M. 
$$
We show that the~sequence $J'_1(u_n)$ is also bounded, what is equivalent to 
$$
\sup\limits_{\norm[{\Wppz{1}{p}{\Omega}}]{v}=1} \abs{\dual{J'_1(u_n)}{v}} < + \infty .
$$
By the~conditions \ref{condition.f1} and \ref{condition.f2} and continuity of $f$ we obtain
$$
\abs{f(t)}\leq c_f \abs{t}^{p-1},\: t\in \setR .
$$
Thus
\begin{align*}
\abs{\int\limits_{\Omega} f(u_n(x))v(x) \dt[x]} &\leq \int\limits_{\Omega}\abs{f(u_n(x))}\abs{v(x)} \dt[x] \\
										&\leq \int\limits_{\Omega} c_f \abs{u_n(x)}^{p-1} \abs{v(x)} \dt[x].
\end{align*}
From H\"older inequality it follows that:
\begin{align*}
\int\limits_{\Omega} \abs{u_n(x)}^{p-1} \abs{v(x)} \dt[x] & \leq \bracket{\int\limits_{\Omega} (\abs{u_n(x)}^{p-1})^{\frac{p}{p-1}}\dt[x]}^{\frac{p-1}{p}} \bracket{\int\limits_{\Omega} \abs{v(x)}^p\dt[x]}^{\frac{1}{p}} \\
&= \norm[{\Lp[p]{\Omega}}]{v} \norm[{\Lp[p]{\Omega}}]{u_n}^{p-1}\\
& \leq c_p^p \norm[{\Wppz{1}{p}{\Omega}}]{v} \norm[{\Wppz{1}{p}{\Omega}}]{u_n}^{p-1} \\
& \leq c_p^p M^{p-1},
\end{align*}
where $c_p$ is the~constant from Lemma \ref{Poincare inequality}. It follows that 
$$
\sup\limits_{\norm[{\Wppz{1}{p}{\Omega}}]{v}=1} \abs{\int\limits_{\Omega} f(u_n(x))v(x)\dt[x]} \leq c_f\, c_p^p\, M^{p-1} ,
$$
and this means that for all $n \in \Nset$
$$
\norm[{\Wpp{-1}{p'}{\Omega}}]{J_1'(u_n)}  \leq c_f\  c_p^p\  M^{p-1} < +\infty.
$$
Thus, the~image of $J_1$ is bounded.
We know that in a~reflexive Banach space each bounded sequence has a~weakly convergent subsequence. So there exists $d \in \Wpp{-1}{p'}{\Omega}$ such that $J_1'(u_n) \weakto d$. 
Suppose, that $\norm[{\Wpp{-1}{p'}{\Omega}}]{J_1'(u_n)-d}>\delta>0$. Then
$$
\exists\: k\: \forall\: (n \geq k)\: \sup\limits_{\norm[{\Wpp{-1}{p'}{\Omega}}]{v}=1} \dual{J_1'(u_n)-d}{v}>\delta>0.
$$
Then we can construct the~sequence $\sequence{v_n}{n=1}{\infty} \subset \Wppz{1}{p}{\Omega}$ such that $\norm[{\Wppz{1}{p}{\Omega}}]{v_n}=1$ and
$$
\exists\: k\: \forall\: (n \geq k)\: \dual{J_1'(u_n)-d}{v_n}>\delta.
$$
Since $\sequence{v_n}{n=1}{\infty}$ is bounded, it admits a~weakly convergent subsequence. By the~Rellich-Kondrachov theorem (Theorem \ref{preliminaries.theorem.rellich-kondrachov}) this sequence admits a~subsequence convergent strongly in $\Lp[p]{\Omega}$. Without any loss at generality we can assume that $v_n$ is weakly convergent in $\Wppz{1}{p}{\Omega}$ and strongly in $\Lp[p]{\Omega}$. The~following holds in an obvious way
\begin{align*}
\dual{J_1'(u_n)-d}{v_n}=\dual{J_1'(u_n)-d}{v}+\dual{J_1'(u_n)}{v_n-v}-\dual{d}{v_n-v}.
\end{align*}
The first and the~third terms converge to zero. Finally
\begin{align*}
0 &\leq \abs{\dual{J_1'(u_n)}{ v_n-v}} = \int\limits_{\Omega} \abs{f(u_n)} \cdot \abs{v_n-v} \dt[x] \leq \int\limits_{\Omega} c_f \abs{u_n}^{p-1}\abs{v_n-v} \dt[x]\\
  & \leq c_f \bracket{\int\limits_{\Omega} \abs{u_n}^p \dt[x]} ^{\frac{p-1}{p}}\bracket{\int\limits_{\Omega} \abs{v_n-v}^p \dt[x]}^{\frac{1}{p}}=c_f \norm[{\Lp[p]{\Omega}}]{u_n}^{p-1} \cdot  \norm[{\Lp[p]{\Omega}}]{v_n-v}\to 0. 
\end{align*}
Thus, we have a~contradiction.
\end{proof}

\noindent The~proof for this fact follows the~steps of proof of Lemma \ref{compact_dericative_J1} almost identically. We prove the~similar statement for $J_2$, as we apply the following compact embedding: $\Wppz{1}{p}{\Omega} \compactembedding \Lp[q]{\Omega}$ instead of $\Wppz{1}{p}{\Omega} \compactembedding \Lp{\Omega}$. 
\begin{lemma}\label{compact_derivative_J2}
Assume \ref{condition.g1}. Then functional $J_2$ has a~compact derivative.
\end{lemma}
\noindent .

\noindent We will use Theorem \ref{Browder-Minty} and Lemma \ref{Chabrowski} to prove the~following lemma:

\begin{lemma}\label{continuous_inverse}
The derivative of operator $\Phi$ admits a~continuous inverse. Namely for 
$$
\dual{\Phi'(u)}{v}=\int\limits_{\Omega} |\nabla u(x)|^{p-2} \nabla u(x) \cdot \nabla v(x) + \mu \frac{|u(x)|^{p-2}u(x)v(x)}{|x|^p}\dt[x] 
$$ 
there exists the~continuous inverse $\bracket{\Phi'}^{-1}\colon \Wpp{-1}{p'}{\Omega}\to \Wppz{1}{p}{\Omega}$.
\end{lemma}

\begin{proof}
We prove that $\Phi'$ is uniformly monotone.
\begin{align*}
&\dual{\Phi'(u_1)-\Phi'(u_2)}{u_1-u_2}\\
=&\int\limits_{\Omega} |\nabla u_1(x)|^{p-2} \nabla u_1(x) \cdot \nabla(u_1(x)- u_2(x)) + \mu \frac{|u_1(x)|^{p-2}u_1(x)(u_1(x)-u_2(x))}{|x|^p}\\
&-|\nabla u_2(x)|^{p-2} \nabla u_2(x) \cdot \nabla (u_1(x)-u_2(x)) - \mu \frac{|u_2(x)|^{p-2}(u_1(x)-u_2(x))}{|x|^p}\dt[x]\\
=&\int\limits_{\Omega}( |\nabla u_1(x)|^{p-2} \nabla u_1(x)- |\nabla u_2(x)|^{p-2} \nabla u_2(x)) \cdot (\nabla u_1(x)- \nabla u_2(x))\\
&+ \mu \frac{(|u_1(x)|^{p-2}u_1(x) |u_2(x)|^{p-2}u_2(x))\cdot (u_1(x)-u_2(x))}{|x|^p}\dt[x].
\end{align*}
By Lemma \ref{Chabrowski} (applied twice)
\begin{align*}
&\dual{\Phi'(u_1)-\Phi'(u_2)}{u_1-u_2}\\
&\geq a_1\int\limits_{\Omega} \bracket{|\nabla u_1(x)-\nabla u_2(x)|^p+\frac{\mu}{|x|^p}|u_1(x)-u_2(x)|^p} \dt[x]\\
&=a_1\bracket{{\norm[{\Wppz{1}{p}{\Omega}}]{u_1-u_2}^p}+\mu {\norm[{\Lp[p]{|x|^{-p},\Omega}}]{u_1-u_2}^p}}\\
&\geq a_1 \norm[{\Wppz{1}{p}{\Omega}}]{u_1-u_2}^{p-1}\cdot \norm[{\Wppz{1}{p}{\Omega}}]{u_1-u_2}.
\end{align*}
One can easily check that  this operator is hemicontinuous. Since it is uniformly monotone operator it is a~monotone and coercive operator. Thus, by Theorem \ref{Browder-Minty} the~assertion holds.
\end{proof}

\section{Existence result}
\noindent In our problem the~role of $X$ space is played by $\Wppz{1}{p}{\Omega}$. It is separable and reflexive Banach space.

\begin{lemma}\label{uniformly_convex}
A space $\bracket{\Wppz{1}{p}{\Omega}, \norm[*]{\cdot}}$ is uniformly convex Banach space.
\end{lemma}

\begin{proof}
We recall that
\begin{align*}
u\mapsto \norm[{*}]{u}=\bracket{p\cdot \Phi(u)}^{\frac{1}{p}}=\sqrt[p]{\int\limits_{\Omega}\abs{\nabla u(x)}^p+\mu \frac{\abs{u(x)}^p}{\abs{x}^p}\dt[x]}
\end{align*}
in a norm on $\Wppz{1}{p}{\Omega}$. We will prove that it
is a~uniformly convex norm. Since we will use two different norms on $\Wppz{1}{p}{\Omega}$, it is essential to add that by $\norm[{\Wppz{1}{p}{\Omega}}]{u}=\sqrt[p]{\int\limits_{\Omega}\abs{\nabla u(x)}^p\dt[x]}$ we understand the~usual norm. In order to prove this, we will use Clarkson's concept of uniformly convex product. By Definition \ref{Clarkson} the~following functional $\Wppz{1}{p}{\Omega}\times \Lp[p]{\abs{x}^{-p},\Omega}\ni (u,v)\mapsto \norm[{**}]{(u,v)}\in \setR_+$, given by the~formula
$$
\norm[{**}]{(u,v)}=\sqrt[p]{\norm[{\Wppz{1}{p}{\Omega}}]{u}^p+\mu \norm[{\Lp[]{\abs{x}^{-p},\Omega}}]{v}^p },
$$
is uniformly convex product; here $\mu$ is the~same constant as in $\Phi$. By the~Clarkson theorem (Theorem \ref{Clarkson}), the~space $\bracket{\Wppz{1}{p}{\Omega}\times \Lp[p]{\abs{x}^{-p},\Omega},\norm[{**}]{\cdot}}$ is a~uniformly convex Banach space. We recall that $\Wppz{1}{p}{\Omega}\hookrightarrow \Lp[]{\abs{x}^{-p},\Omega}$. Thus, we can easily observe that $\norm[{**}]{(u,u)}=\norm[{*}]{u}$ for $u \in \Wppz{1}{p}{\Omega}$.

We will prove now that $u\mapsto \norm[{*}]{u}$ is an uniformly convex norm. Let $2>\epsilon>0, \norm[{*}]{u}=\norm[{*}]{v}=1$ and $\norm[{*}]{u-v}\geq \epsilon.$ Whence $\norm[{*}]{(u,u)}=1$ and $\norm[{**}]{(v,v)}=1$ as well as $\norm[{**}]{(u,u)-(v,v)}$ $\geq \epsilon$.
Thus, by the~uniform convexity of $\bracket{\Wppz{1}{p}{\Omega} \times \Lp[]{\abs{x}^{-p}, \Omega}; \norm[{**}]{\cdot}}$,
we get that there exists $\delta_{\epsilon} \in (0,1)$ such that
$$
1-\delta_{\epsilon}\geq \norm[{**}]{\frac{(u,u)+(v,v)}{2}}=\norm[{**}]{\bracket{\frac{u+v}{2},\frac{u+v}{2}}}=\norm[{*}]{\frac{u+v}{2}},
$$
which proves that $\bracket{\Wppz{1}{p}{\Omega}, \norm[{*}]{\cdot}}$ is uniformly convex Banach space.
\end{proof} 

\begin{definition}\cite{Ricceri}
By $W_X$ we denote the~class of all functionals $\Phi\colon X \to \setR$ possessing the~following property: if $\set{u_n}$ is a~sequence in $X$ converging weakly to $u\in X$ and $$\liminf\limits_{n\to\infty} \Phi (u_n)\leq \Phi(u),$$ then $\set{u_n}$ has a~ subsequence converging strongly to $u$.
\end{definition}

\begin{lemma}
$\Phi$ belongs to $W_{\Wppz{1}{p}{\Omega}}$. 
\end{lemma}

\begin{proof}
Assume $\bracket{u_n}^{\infty}_{n=1}\subset \Wppz{1}{p}{\Omega}$, $u_n\weakto u$ and $\liminf\limits_{n\to \infty} \Phi(u_n)\leq \Phi(u)$. From sequentially weakly lower semicontinuity of functional $\Phi$ it follows $\liminf\limits_{n\to \infty} \Phi(u_n)\geq\Phi(u)$. 
Thus there exists a~subsequence $u_{n_k}\in \Wppz{1}{p}{\Omega}$ of $u_n$ such that $\lim\limits_{n\to \infty}\Phi(u_{n_k})=\Phi(u)$. Explicitly:
\begin{align*}
\lim\limits_{n\to \infty} \bracket{\frac{1}{p}\int\limits_{\Omega}\abs{\nabla u_{n_k}(x)}^p +\mu \frac{\abs{u_{n_k}}^p}{\abs{x}^p}\dt[x]}&=\frac{1}{p}\int\limits_{\Omega}\abs{\nabla u(x)}^p+\mu \frac{\abs{u(x)}^p}{\abs{x}^p} \dt[x]=\frac{1}{p}\norm[{*}]{u}^p;\\
\lim\limits_{n\to \infty} \frac{1}{p}\norm[{*}]{u_{n_k}}^p&=\frac{1}{p}\norm[{*}]{u}^p;\\
\lim\limits_{n\to \infty} \norm[{*}]{u_{n_k}}&=\norm[{*}]{u}.
\end{align*} 
We proved that $\Wppz{1}{p}{\Omega}$ is uniformly convex Banach space, therefore there exists a~subsequence $(u_{n_k})$ of $(u_n)$ such that $\norm[{*}]{u_{n_k}-u}\to 0$.
\end{proof}

\begin{lemma}\label{sup>0}
Assume, that \ref{condition.f1}-\ref{condition.f3} holds. Then $\sup\limits_{\Phi (u)>0} \frac{J_1(u)}{\Phi(u)}>0$. 
\end{lemma}

\begin{proof}
We show equivalently there exists $u_\delta$ such that $\frac{J_1(u_{\delta})}{\Phi (u_{\delta})}>0$. 
By \ref{condition.f3} there exists $s_0\in \setR$ such that $F(s_0)>0$. Let $\delta \in (0,1)$. Then there exists 
$ x_0 \in \Omega\setminus\partial\Omega $ such that there exist $R$, $r$ such that $R>r>0$ with $B(x_0,R)\subset \Omega\setminus\partial\Omega $.
We can choose $u_{\delta}\in \Wppz{1}{p}{\Omega}$ such that 
\begin{enumerate}
\item $\text{supp }u_{\delta} \subset B(x_0,R)\subset \Omega,$
\item $u_{\delta} \mid_{B(x_0,r+\delta(R-r))}\equiv s_0,$
\item ${\norm{u_{\delta}}}_{\infty}\leq s_0. $
\end{enumerate}
Then
\begin{align*}
J_1(u_{\delta})=& \int\limits_{B(x_0,R)}F(u_{\delta}(x))dx\geq F(s_0) \cdot \text{Vol}(N)[r+\delta(R-r)]^N\\
				& -\max\limits_{t\in [-s_0, s_0]} \abs{F(t)} \bracket{\text{Vol}(N)[R^N-(r+\delta(R-r))]^N},
\end{align*}
where $\text{Vol}(N)$ stands for the~measure of $N$-dimensional unit ball/
As $\delta \to 1$, then $J_1(u_{\delta})\to C>0.$ Let $\bar{\delta}$ be such that $J_1(u_{\bar{\delta}})>0$. Observe, that $u_{\bar{\delta}}\not\equiv 0$. So $\Phi(u_{\bar{\delta}})>0$ and $J_1(u_{\delta})>0$.
\end{proof}

\begin{lemma}\label{lim_u_to_0_<0}
Assume, that \ref{condition.f1}-\ref{condition.f3} holds. Then $\limsup\limits_{\norm[\Wppz{1}{p}{\Omega}]{u} \to 0}\frac{J_1(u)}{\Phi(u)}\leq 0$.
\end{lemma}

\begin{proof}
Let $\epsilon >0$ and $\eta \in \bracket{0,\frac{p^2}{n-p}}$. We will use conditions \ref{condition.f1} and \ref{condition.f2}. For $\epsilon_{f_1}=\frac{1}{(C_p)^p}\epsilon $ there exists by \ref{condition.f1} and \ref{condition.f2} $\delta_{\epsilon}>0$ such that 
\begin{align*}
\abs{f(t)}  \leq \frac{1}{(C_p)^p} \epsilon \abs{t}^{p-1} \text{   for  } & \abs{t} \in [0,\delta_{\epsilon}] \cup [\delta_{\epsilon}^{-1},+\infty).\\
\abs{f(t)}  \leq M_{\epsilon} \abs{t}^{p-1+\eta} \text{   for   } & \abs{t} \in (\delta_{\epsilon}, \delta_{\epsilon}^{-1}).
\end{align*}
Then for all $t\in \setR$  
$$
\abs{f(t)} \leq \frac{1}{(C_p)^p}\epsilon \abs{t}^{p-1} + M_{\epsilon} \abs{t}^{p-1+\eta}. 
$$
\begin{align*}
\abs{J_1(u)}&=\abs{\int\limits_{\Omega} \int\limits^{u(x)}_{0} f(t) \dt[t]\ \dt[x]} \leq \int\limits_{\Omega}\abs{\int\limits^{u(x)}_{0}\abs{f(t)} \dt[t]} \dt[x] \leq \int\limits_{\Omega}\abs{\int\limits^{u(x)}_{0} \frac{1}{(C_p)^p} \epsilon \abs{t}^{p-1} + M_{\epsilon} \abs{t}^{p-1+\eta} \dt[t] }\dt[x]\\
& \leq \int\limits_{\Omega} \frac{1}{(C_p)^p} \epsilon \frac{1}{p} \abs{u(x)}^p + M_{\epsilon} \frac{1}{p+\eta} \abs{u(x)}^{p+\eta} \dt[x]\\
& = \frac{1}{(C_p)^p} \epsilon \frac{1}{p}\int\limits_{\Omega} \abs{u(x)}^p \dt[x] + M_{\epsilon} \frac{1}{p+\eta} \int\limits_{\Omega} \abs{u(x)}^{p+\eta} \dt[x]\\
& = \frac{1}{(C_p)^p}\epsilon \frac{1}{p} \norm[{\Lp[p]{\Omega}}]{u}^p + M_\epsilon \frac{1}{p+\eta} \norm[{\Lp[p+\eta]{\Omega}}]{u}^{p+\eta}.
\end{align*}
From Theorem \ref{preliminaries.theorem.rellich-kondrachov}, for $\eta \in \bracket{0,\frac{p^2}{n-p}}$ it follows $\Wppz{1}{p}{\Omega} \compactembedding \Lp[p+\eta]{\Omega}$.
We know that $\Phi(u) \geq \frac{1}{p} \norm[{\Wppz{1}{p}{\Omega}}]{u}^p$, so by dividing both sides we obtain
\begin{align*}
\frac{J_1(u)}{\Phi(u)} &\leq \frac{\frac{1}{(C_p)^p} \epsilon \frac{1}{p} \norm[{\Lp[p]{\Omega}}]{u}^p + M_{\epsilon} \frac{1}{p+\eta} \norm[{\Lp[p+\eta]{\Omega}}]{u}^{p+\eta}}{\frac{1}{p} \norm[{\Wppz{1}{p}{\Omega}}]{u}^p}
\leq \frac{\epsilon\frac{1}{p} \norm[{\Wppz{1}{p}{\Omega}}]{u}^p + \overline{M_{\epsilon}} \norm[{\Wppz{1}{p}{\Omega}}]{u}^{p+\eta}}{\frac{1}{p}  \norm[{\Wppz{1}{p}{\Omega}}]{u}^p}\\
&\leq \epsilon + \overline{M_{\epsilon}}\ p \norm[{\Wppz{1}{p}{\Omega}}]{u}^\eta.
\end{align*}
Since $\norm[{\Wppz{1}{p}{\Omega}}]{u}\to 0 $ thus $\limsup\limits_{\norm[{\Wppz{1}{p}{\Omega}}]{u} \to 0}\frac{J_1(u)}{\Phi(u)}\leq \epsilon$. Since $\epsilon>0$ was chosen arbitrarily we obtain $\limsup\limits_{\norm[{\Wppz{1}{p}{\Omega}}]{u} \to 0}\frac{J_1(u)}{\Phi(u)}\leq 0$.
\end{proof}

\begin{lemma}\label{lim_u_to_infty_<0}
Assume, that \ref{condition.f1}-\ref{condition.f3} holds. Then $\limsup\limits_{\norm[{\Wppz{1}{p}{\Omega}}]{u} \to \infty}\frac{J_1(u)}{\Phi(u)}\leq 0$.
\end{lemma}

\begin{proof}
Let $\epsilon >0$ and $\eta \in (0, \min\{\frac{p^2}{n-p}, p-1\})$. We will use conditions \ref{condition.f1} and \ref{condition.f2}. For $\epsilon_{f_1}=\frac{1}{(C_p)^p}\epsilon $ there exists by \ref{condition.f1} and \ref{condition.f2} $\delta_{\epsilon}$ and there exists $\eta \in \bracket{0,\frac{p^2}{n-p}}$ such that 
\begin{align*}
\abs{f(t)} \leq \frac{1}{(C_p)^p} \epsilon \abs{t}^{p-1} \text{   for   }& \abs{t} \in [0,\delta_{\epsilon}] \cup [\delta_{\epsilon}^{-1},+\infty).\\
\abs{f(t)} \leq M_{\epsilon} \abs{t}^{p-1-\eta} \text{   for   }& t \in (\delta_{\epsilon}, \delta_{\epsilon}^{-1}).
\end{align*}
Then
\begin{equation*}
\abs{f(t)} \leq \frac{1}{(C_p)^p}\epsilon \abs{t}^{p-1}+M_{\epsilon} \abs{t}^{p-1-\eta},  t \in \setR.
\end{equation*}
\begin{align*}
\abs{J_1(u)}&=\abs{\int\limits_{\Omega} \int\limits^{u(x)}_{0} f(t) \dt[t]\ \dt[x]} \leq \int\limits_{\Omega}\abs{\int\limits^{u(x)}_{0}\abs{f(t)} \dt[t]} \dt[x] \leq \int\limits_{\Omega}\abs{\int\limits^{u(x)}_{0} \frac{1}{(C_p)^p} \epsilon \abs{t}^{p-1} + M_{\epsilon} \abs{t}^{p-1-\eta} \dt[t] } \dt[x] \\
&\leq \int\limits_{\Omega} \frac{1}{(C_p)^p} \epsilon \frac{1}{p} \abs{u(x)}^p + M_{\epsilon} \frac{1}{p-\eta} \abs{u(x)}^{p-\eta} \dt[x]\\
&=\frac{1}{(C_p)^p} \epsilon \frac{1}{p}\int\limits_{\Omega} \abs{u(x)}^p \dt[x] + M_{\epsilon} \frac{1}{p-\eta} \int\limits_{\Omega} \abs{u(x)}^{p-\eta}\dt[x]\\
&= \frac{1}{(C_p)^p}\epsilon \frac{1}{p} \norm[{\Lp[p]{\Omega}}]{u}^p + M_\epsilon \frac{1}{p-\eta} \norm[{\Lp[p-\eta]{\Omega}}]{u}^{p-\eta}.
\end{align*}
We know that $\Phi(u) \geq \frac{1}{p} \norm[{\Wppz{1}{p}{\Omega}}]{u}$, so by dividing both sides we obtain
\begin{align*}
\frac{J_1(u)}{\Phi(u)} &\leq \frac{\frac{1}{(C_p)^p} \epsilon \frac{1}{p} \norm[{\Lp[p]{\Omega}}]{u}^p + M_{\epsilon} 
\frac{1}{p-\eta} \norm[{\Lp[p-\eta]{\Omega}}]{u}^{p-\eta}}{\frac{1}{p} {\norm[{\Wppz{1}{p}{\Omega}}]{u}}^p}
\leq \frac{\epsilon\frac{1}{p} \norm[{\Wppz{1}{p}{\Omega}}]{u}^p + \overline{M_{\epsilon}} \norm[{\Wppz{1}{p}{\Omega}}]{u}^{p-\eta}}{\frac{1}{p}  \norm[{\Wppz{1}{p}{\Omega}}]{u}^p}\\
&\leq \epsilon + \overline{M_{\epsilon}}\ p \norm[{\Wppz{1}{p}{\Omega}}]{u}^{-\eta}.
\end{align*}
Since $\norm[{\Wppz{1}{p}{\Omega}}]{u}\to \infty $ thus $\limsup\limits_{\norm[{\Wppz{1}{p}{\Omega}}]{u} \to \infty}\frac{J_1(u)}{\Phi(u)}\leq \epsilon$. \\
Since $\epsilon>0$ chosen arbitrarily thus $\limsup\limits_{\norm[{\Wppz{1}{p}{\Omega}}]{u} \to \infty}\frac{J_1(u)}{\Phi(u)}\leq 0$.
\end{proof}

We will use the~following Ricceri theorem to show that the~Problem \ref{problem.weak.solution} has a~solution.

\begin{theorem}[\textbf{Ricceri}\cite{Ricceri}]\label{Ricceri}\mbox{}\\
Let $X$ be a~separable and reflexive real Banach space; $\Phi \colon X \to \setR$ a~coercive, s.w.l.s.c. $\Ck[1]{}$~functional, belonging to $W_X$, bounded on each bounded subset of $X$ and whose derivative admits a~continuous inverse on $\dualSpace{X}$; $J_1 \colon X \to \setR$ a~$\Ck[1]{}$ functional with compact derivative. Assume that $\Phi$ has a~strict local minimum $x_0$ with $\Phi(x_0)=J_1(x_0)=0$. Finally, setting 
\begin{align*}
& \alpha = \max \set{0, \limsup_{\norm{x}\to +\infty} \frac{J_1(x)}{\Phi(x)}, \limsup_{x \to x_0} \frac{J_1(x)}{\Phi(x)}}\\
& \beta = \sup\limits_{x \in \Phi^{-1}((0,+\infty))} \frac{J_1(x)}{\Phi(x)},
\end{align*}
assume $\alpha < \beta$. Then for each compact interval $[a,b] \subset (\frac{1}{\beta}, \frac{1}{\alpha})$, with the~conventions $(\frac{1}{0} = +\infty, \frac{1}{+\infty}=0)$, there exists $ r>0$ with the~following property: for every $\lambda \in [a,b]$ and every $\Ck[1]{}$~functional $J_2 \colon X \to \setR$ with compact derivative, there exists $\delta >0$ such that, for each $\mu \in [0, \delta]$, the~equation
\begin{equation*}
\Phi'(x) = \lambda J_1'(x) + \mu J_2'(x),
\end{equation*}
has at least three solutions whose norm are less than $r$.
\end{theorem}

\begin{theorem}[The existence of three weak solutions of Problem \ref{problem.weak.solution}] \label{main.theorem}
Assume that conditions \ref{condition.f1}-\ref{condition.f3} and \ref{condition.g1} hold. Then there exists $\beta>0$ such that for each compact interval $[a,b] \subset (\beta, +\infty)$, there exists $ r>0$ with the~following property: for every $\lambda \in [a,b]$, there exists $\delta >0$ such that, for each $\gamma \in [0, \delta]$, the~Problem \ref{problem.weak.solution} has at least three solutions whose norm are less than $r$. 
\end{theorem}

\begin{proof}
$\Wppz{1}{p}{\Omega}$ is obviously separable and reflexive. By Theorem \ref{Ricceri}  and Lemmas \ref{well_defined_functionals}, \ref{swlsc}, \ref{coercive}, \ref{compact_dericative_J1}, \ref{compact_derivative_J2}, \ref{continuous_inverse}, \ref{uniformly_convex}, \ref{sup>0}, \ref{lim_u_to_0_<0}, \ref{lim_u_to_infty_<0} and since $\Phi(0)=J_1(0)=0$ and $0 \in \Wppz{1}{p}{\Omega}$ is a~strict minimum of $\Phi$. By abstract existence result of Ricceri there exists $[a,b]\subset \bracket{\frac{1}{\chi},\frac{1}{\tau}}$, such that for all $J_2\in C^1$, $J_2$ has a~compact derivative and Problem \ref{problem.weak.solution} has 3 solutions. So two of them must be non trivial.
\end{proof}

\section{Example}
Let $\Omega\subseteq \setR^n$ be bounded set with Lipschitz boundary such that $0 \in \Omega$. Furthermore let $2 \leq p < n$ and $\mu \in(0,+\infty)$. We consider the following problem:
\begin{problem2} \label{example.problem}
Find $u \in \Wppz{1}{p}{\Omega}$ such that for all $v \in \Wppz{1}{p}{\Omega}$
\begin{align*}
\int\limits_{\Omega} \abs{\nabla u(x)}^{p-2} \nabla u(x) \nabla v(x) + \mu \frac{\abs{u(x)}^{p-2} u(x) v(x) }{ \abs{x}^p} \dt[x] \\
=\lambda \int\limits_{\Omega} f(u(x))v(x) \dt[x] + \gamma \int\limits_{\Omega} g(u(x))v(x) \dt[x]
\end{align*}
with functions $f,g\colon \setR \to \setR$ given by
\begin{align*}
f(t)=\left\lbrace
\begin{array}{cl}
\bracket{\frac{\pi}{2r}}^{p-1}\abs{\sin(rt)}^p & \text{for } \abs{t} \leq \frac{\pi}{2r}\vspace{3mm}\\
\frac{\bracket{1+\bracket{\frac{\pi}{2r}}^2}\abs{t}^{p-1}}{1+t^2} & \text{for } \abs{t} > \frac{\pi}{2r}
\end{array}
\right.
\end{align*}
and
\begin{align*}
g(t)=\left\lbrace
\begin{array}{cl}
1+\abs{t}^{q-1} & \text{for } \abs{t} \leq z\vspace{3mm}\\
\frac{\bracket{1+z^2}\bracket{1+z^{q-1}}}{1+t^2} & \text{for } \abs{t}>z,
\end{array}
\right. 
\end{align*}
where $r,z>0$ are fixed constants and $1<q<\frac{np}{n-p}$. 
\end{problem2}
\noindent Both functions $f$ and $g$ are obviously continuous on $\setR$. We will show that function $f$ satisfies conditions \ref{condition.f1}-\ref{condition.f3}. Indeed, 
\begin{align*}
\lim\limits_{\abs{t}\to 0}\frac{f(t)}{\abs{t}^{p-1}}=\lim\limits_{\abs{t}\to 0} \frac{\bracket{\frac{\pi}{2r}}^{p-1}\abs{\sin(rt)}^p}{\abs{t}^{p-1}}=\lim\limits_{\abs{t}\to 0} \bracket{\frac{\pi}{2}}^{p-1}\abs{\frac{\sin(rt)}{rt}}^{p-1} \abs{\sin(rt)} =0.
\end{align*}
Hence condition \ref{condition.f1} holds.
Moreover, also condition \ref{condition.f2} is verified, because
\begin{align*}
\lim\limits_{\abs{t}\to \infty}\frac{f(t)}{\abs{t}^{p-1}}=\lim\limits_{\abs{t}\to \infty} \frac{\bracket{1+\bracket{\frac{\pi}{2r}}^2}\abs{t}^{p-1}}{\bracket{1+t^2}\abs{t}^{p-1}}=0.
\end{align*}
Furthermore $\sup\limits_{t\in \setR} F(t)>0$, because for $t=\frac{\pi}{2r}$ we have
\begin{align*}
F\bracket{\frac{\pi}{2r}}&=\int\limits_{0}^\frac{\pi}{2r} f(s) ds= \int\limits_{0}^\frac{\pi}{2r} \bracket{\frac{\pi}{2r}}^{p-1}\abs{\sin(rs)}^p \dt[s]=\bracket{\frac{\pi}{2r}}^{p-1}\int\limits_{0}^\frac{\pi}{2r} \abs{\sin(rs)}^p \dt[s]\\
&\geq \bracket{\frac{\pi}{2r}}^{p-1}\int\limits_{\frac{\pi}{4r}}^\frac{\pi}{2r} \bracket{\frac{\sqrt{2}}{2}}^p \dt[s]=\frac{1}{2}\bracket{\frac{\pi\sqrt{2}}{4r}}^p>0.
\end{align*}
Also condition \ref{condition.g1} is satisfied by function $g$ with constant $c_g= 1$. For $\abs{t}\leq r$ we obviously have
\begin{align*}
g(t)=1+\abs{t}^{q-1}.
\end{align*}
On the other hand for $\abs{t}\geq r$
\begin{align*}
g(t)= \frac{\bracket{1+z^2}\bracket{1+z^{q-1}}}{1+t^2}\leq \bracket{1+z^{q-1}} \leq \bracket{1+\abs{t}^{q-1}}.
\end{align*}
From theorem \ref{main.theorem}, there exists $\beta>0$ such that for each compact interval $[a,b] \subset (\beta, +\infty)$, there exists $ w>0$ with the~following property: for every $\lambda \in [a,b]$, there exists $\delta >0$ such that, for each $\gamma \in [0, \delta]$, the~Problem \ref{example.problem} has at least three solutions whose norm are less than $w$.
\section*{Acknowledgement}
The authors would like to thank professor Giovanni Molica Bisci from Universit\'a degli Studi Mediterranea in Italy for introducing authors to this type of problems and to thank professor Aleksander \'{C}wiszewski from Nicolaus Copernicus University in Poland, for his remarks that allowed to exclude unnecessary assumptions and include the omitted ones.
\bibliographystyle{plain}
\bibliography{bibliografia}

\begin{thebibliography}{10}

\bibitem{Adams}
R.~A. Adams and J.~J.~F. Fournier.
\newblock {\em Sobolev Spaces}.
\newblock Elsevier, 2003.

\bibitem{Azorero}
J.~P.~Garc\'ia Azorero and I.~Peral Alonso.
\newblock Hardy inequalities and some critical elliptic and parabolic problems.
\newblock {\em Journal of Differential Equations}, (144):441--476, 1998.

\bibitem{Binding.Drabek.Huang}
P.~A. Binding, P.~Dr\'{a}bek, and Y.~X. Huang.
\newblock Existence of multiple solutions of critical quasilinear elliptic
  {N}eumann problems.
\newblock {\em Nonlinear Analysis}, 42:613--629, 2000.

\bibitem{Bonanno.Candito}
G.~Bonanno and P.~Candito.
\newblock Three solutions to a {N}eumann problem for elliptic equations
  involving the p-{L}aplacian.
\newblock {\em Archiv der Mathematik}, (80):424--429, 2003.

\bibitem{Chabrowski}
J.~H. Chabrowski.
\newblock {\em Variational methods for potential operator equations: with
  applications to nonlinear elliptic equations}, volume~24.
\newblock Walter de Gruyter, 1997.

\bibitem{Clarkson}
J.~A. Clarkson.
\newblock Uniformly convex spaces.
\newblock {\em Transactions of the American Mathematical Society},
  (3):396--414, 1936.

\bibitem{Cuesta.Quoirin}
M.~Cuesta and H.~R. Quoirin.
\newblock A weighted eigenvalue problem for the p-{L}aplacian plus a potential.
\newblock {\em Nonlinear Differential Equations and Applications NoDEA},
  16(4):469--491, 2009.

\bibitem{Dagui.Molica-Bisci}
G.~D'Aguì and G.~Molica Bisci.
\newblock Existence results for an elliptic {D}irichlet problem.
\newblock {\em Le Matematiche}, (66):133--141, 2011.

\bibitem{Ferrara.Molica-Bisci}
M.~Ferrara and G.~Molica Bisci.
\newblock Existence result for elliptic problems with {H}ardy potential.
\newblock {\em Bulletin des Sciences Math{\'e}matiques}, 138(7):846--859, 2014.

\bibitem{Filippucci.Pucci.Robert}
R.~Filippucci, P.~Pucci, and F.~Robert.
\newblock On a p-{L}aplace equation with multiple critical nonlinearities.
\newblock {\em Journal de math{\'e}matiques pures et appliqu{\'e}es},
  91(2):156--177, 2009.

\bibitem{Gasinski.Papageorgiou}
L.~Gasiński and N.~S. Papageorgiou.
\newblock Nontrivial solutions for a class of resonant p-{L}aplacian {N}ewmann
  problems.
\newblock {\em Nonlinear Analysis}, 71:6365--6372, 2009.

\bibitem{Kristaly.Varga}
A.~Krist{\'a}ly and C.~Varga.
\newblock Multiple solutions for elliptic problems with singular and sublinear
  potentials.
\newblock {\em Proceedings of the American Mathematical Society},
  (7):2121--2126, 2007.

\bibitem{Lindqvist}
P.~Lindqvist.
\newblock On the equation $\text{div}(|\nabla u|^{p-2}\nabla u)+\lambda
  |u|^{p-2}u=0$.
\newblock {\em P.Amer. Math. Soc.}, pages 157--164, 1990.

\bibitem{Lucia.Prashanth}
M.~Lucia and S.~Prashanth.
\newblock Simplicity of principal eigenvalue for p-{L}aplace operator with
  singular indefinite weight.
\newblock {\em Archiv der Mathematik}, 86(1):79--89, 2006.

\bibitem{Ricceri}
B.~Ricceri.
\newblock A further three critical point theorem.
\newblock {\em Nonlinear Analysis}, (71):4151--4157, 2009.

\bibitem{Rudin}
W.~Rudin.
\newblock {\em Functional Analysis}.
\newblock PWN, 2011.

\bibitem{Zeidler}
E.~Zeidler.
\newblock {\em Nonlinear Functional Analysis and its Applications}, volume
  II/B.
\newblock Springer, 1990.

\end{thebibliography}

\end{document}